\documentclass[]{amsart}
\usepackage[latin2]{inputenc}
\usepackage{t1enc}
\usepackage[mathscr]{eucal}
\usepackage{graphicx}
\usepackage{graphics}
\numberwithin{equation}{section}
\usepackage{hyperref} 

\hypersetup{
 colorlinks=true,       
    linkcolor=cyan,          
    citecolor=green,        
    filecolor=magenta,      
   urlcolor=cyan           
   }

\allowdisplaybreaks
\newcommand{\eq}{\begin{equation}}
\newcommand{\en}{\end{equation}}
\newcommand{\lb}[1]{\label{#1}}

\DeclareSymbolFont{bbold}{U}{bbold}{m}{n}
\DeclareSymbolFontAlphabet{\mathbbold}{bbold}
\newcommand{\ind}{\mathbbold{1}}

\makeatletter

\newcommand{\Rmnum}[1]{\expandafter\@slowromancap\romannumeral #1@}
\makeatother

\newtheorem{thm}{Theorem}[section]
\newtheorem{prop}[thm]{Proposition}
\newtheorem{lem}[thm]{Lemma}

\theoremstyle{definition}
\newtheorem{remark}[thm]{Remark}
\newtheorem{defn}[thm]{Definition}

\newtheorem{notation}[thm]{Notation}

\numberwithin{equation}{section}

\numberwithin{figure}{section}

\newcommand{\reals}{\mathbb{R}}

\newcommand{\nats}{\mathbb{N}}

\newcommand{\lra}{\leftrightarrow}
\newcommand{\ra}{\rightarrow}
\renewcommand{\P}{\mathbb{P}}
\newcommand{\Q}{\mathbb{Q}}
\newcommand{\E}{\mathbb{E}}
\newcommand{\ed}{\,{\buildrel d \over =}\,}

\renewcommand{\and}{ \quad \text{and} \quad }

\begin{document}

\title[]{Two continua of embedded regenerative sets}

\author{Steven N. Evans}
\address{Department of Statistics \#3860 \\ 367 Evans Hall \\ University of California at Berkeley \\ Berkeley, CA 94720-3860 \\ USA}
\email{evans@stat.berkeley.edu}
\thanks{SNE supported in part by NSF grant DMS-1512933.}

\author{Mehdi Ouaki}
\address{Department of Statistics \#3860 \\ 367 Evans Hall \\ University of California at Berkeley \\ Berkeley, CA 94720-3860 \\ USA}
\email{mouaki@berkeley.edu}
\thanks{}



\begin{abstract}
Given a two-sided real-valued L\'evy process $(X_t)_{t \in \reals}$, 
define processes  $(L_t)_{t \in \reals}$ and $(M_t)_{t \in \reals}$ 
by 
$L_t := \sup\{h \in \reals : h - \alpha(t-s) \le X_s \text{ for all } s \le t\} 
= \inf\{X_s + \alpha(t-s) : s \le t\}$, $t \in \reals$,
and 
$M_t  := \sup \{ h \in \reals : h - \alpha|t-s| \leq X_s  \text{ for all } s \in \reals \} 
= \inf \{X_s + \alpha |t-s| : s \in \reals\}$, $t \in \reals$.
The corresponding contact sets are the random sets 
$\mathcal{H}_\alpha := \{ t \in \reals : X_{t}\wedge X_{t-} = L_t\}$
and
$\mathcal{Z}_\alpha := \{ t \in \reals : X_{t}\wedge X_{t-} = M_t\}$.
For a fixed $\alpha>\E[X_1]$ (resp. $\alpha>|\E[X_1]|$) the set $\mathcal{H}_\alpha$ (resp. $\mathcal{Z}_\alpha$) is non-empty, closed, unbounded above and below, stationary, and regenerative. 
The collections 
$(\mathcal{H}_{\alpha})_{\alpha > \E[X_1]}$ 
and 
$(\mathcal{Z}_{\alpha})_{\alpha > |\E[X_1]|}$
are increasing in $\alpha$ and the regeneration property is compatible with these inclusions in that each family is a continuum of embedded regenerative sets in the sense of Bertoin. We show that
$(\sup\{t < 0 : t \in \mathcal{H}_\alpha\})_{\alpha > \E[X_1]}$ 
is a c\`adl\`ag, nondecreasing, pure jump process with independent increments
and determine the intensity measure of the associated Poisson process of jumps.
We obtain a similar result for 
$(\sup\{t < 0 : t \in \mathcal{Z}_\alpha\})_{\alpha > |\beta|}$
when  $(X_t)_{t \in \reals}$ is a (two-sided) Brownian motion with drift $\beta$.
\end{abstract}

\maketitle

\section{Introduction}

Let $X = (X_t)_{t \in \reals}$ be a two-sided, real-valued L\'evy process on a complete probability space $(\Omega, \mathcal{F}, \mathbb{P})$.  
That is, $X$ has c\`adl\`ag paths and stationary, independent increments. Assume that $X_0=0$.  
Let $(\mathcal{F}_t)_{t \in \reals}$ be the natural filtration of $X$ augmented by the $\mathbb{P}$-null sets. 
Suppose that $\E[X_1^{+}]<\infty$ so that $\E[X_1]$ is well-defined (but possibly $-\infty$).

For $\alpha > \E[X_1]$ define a process  $(L_t)_{t \in \reals}$ by 
\[
L_t := \sup\{h \in \reals : h - \alpha(t-s) \le X_s \text{ for all } s \le t\} 
= \inf\{X_s + \alpha(t-s) : s \le t\},  \quad t \in \reals,
\]
and set 
\[
\mathcal{H}_\alpha := \{ t \in \reals : X_{t}\wedge X_{t-} = L_t\}.
\]
Equivalently,
\[
\mathcal{H}_{\alpha}:=\left\{ t \in \reals : X_{t}\wedge X_{t-}-\alpha t = \inf_{u \le t} ( X_u - \alpha u ) \right\}.
\]

By the strong law of large numbers for L\'evy processes (see, for example, \cite[Example 7.2]{kyprianou2006introductory})
\[
\lim_{t \rightarrow +\infty} \frac{X_t}{t}= \lim_{t \rightarrow -\infty} \frac{X_t}{t} = \E[X_1] \quad \text{a.s.}
\]
so that
\[
\lim_{t \rightarrow +\infty} X_t-\alpha t =-\infty \quad \text{a.s.}
\]
and
\[
\lim_{t \rightarrow -\infty} X_t-\alpha t =+\infty \quad \text{a.s.}
\]
It follows from Lemma~\ref{deterministic} below that $\mathcal{H}_\alpha$ is almost surely a non-empty, closed set that is unbounded above and below.

We show in Theorem~\ref{reg:thm} that $\mathcal{H}_\alpha$ is a {\em regenerative set} in the sense of \cite{fitztaksar}.
Moreover, we observe in Lemma~\ref{L:H_alpha_increasing} that
\[
\mathcal{H}_{\alpha_1} \subseteq \mathcal{H}_{\alpha_2} \subseteq \dots \subseteq \mathcal{H}_{\alpha_n}
\]
for
\[
\E[X_1]<\alpha_1<\alpha_2< \dots <\alpha_n.
\]
In Section~\ref{S:reg_emb_gen} we recall from \cite{bertoin99} the notion of {\em regenerative embeddings}
and establish in Proposition~\ref{P:H_alpha_reg_emb} that these embeddings are regenerative.
As a consequence, we derive the following result which we prove in Section~\ref{S:H_alpha_reg_emb}.

\begin{thm}
\lb{ind:onesided}
For $\alpha> \E[X_1]$ set
\[
G_{\alpha}:=\sup \{t<0 : t \in \mathcal{H}_{\alpha}\}.
\]
Then $(G_{\alpha})_{\alpha > \E[X_1]}$ is a nondecreasing, c\`adl\`ag,  pure jump process with independent increments. 
The point process
\[
\{ (\alpha, G_{\alpha}-G_{\alpha-}): G_{\alpha}-G_{\alpha-}>0\}
\]
is a Poisson point process on $(\E[X_1],\infty) \times (0,\infty)$ with intensity measure
\[
 \gamma(dx \times dt)=t^{-1}\mathbb{P}\left\{\frac{X_t}{t} \in dx\right\} \, dt \; \ind_{\{t >0, \, x > \E[X_1]\}}.
\]
\end{thm}

The set $\mathcal{H}_\alpha$ is obviously closely related to the {\em ladder time} 
\[
\mathcal{R}_{\alpha}:= \left\{ t \in \reals : X_t-\alpha t =\inf_{u \le t} (X_u-\alpha u)\right\}
\]
of the L\'evy process $(X_t-\alpha t)_{t \in \reals}$.
We clarify the connection with the following result which is proved in Section~\ref{S:ladder_times}.

\begin{prop}
\lb{P:ladder_times}
The following hold almost surely.
\begin{itemize}
\item[(i)]
$\mathcal{R}_{\alpha} \subseteq \mathcal{H}_\alpha$.
\item[(ii)]
$\mathcal{R}_{\alpha}$ is closed from the right.
\item[(iii)]
$\mathbf{cl}(\mathcal{R}_{\alpha})=\mathcal{H}_{\alpha}$.
\item[(iv)]
$\mathcal{H}_{\alpha} \setminus \mathcal{R}_{\alpha}$ consists of points in $\mathcal{H}_{\alpha}$ that are isolated on the right
and so, in particular, this set is countable.
\end{itemize}
\end{prop}

Given $\alpha>0$, denote by $(M_t)_{t \in \reals}$ be the $\alpha$-Lipschitz minorant of the two-sided L\'evy process $(X_t)_{t \in \reals}$;
that is, $t \mapsto M_t$ is the greatest $\alpha$-Lipschitz function dominated by $t \mapsto X_t$ (our notation suppresses the dependence of $M$
on $\alpha$).
We refer the reader to \cite{zbMATH06288068} and \cite{1905.07038} for extensive investigations of the Lipschitz minorant of a L\'evy process.
The $\alpha$-Lipschitz minorant exists if
\begin{equation*}
\lb{condi}
\E[\vert X_1 \vert]< \infty \text{ and } \alpha > \vert \E[X_1] \vert
\end{equation*}
and we suppose that these conditions hold when discussing $(M_t)_{t \in \reals}$.
Then, 
\[
M_t  = \sup \{ h \in \reals : h - \alpha|t-s| \leq X_s  \text{ for all } s \in \reals \} 
= \inf \{X_s + \alpha |t-s| : s \in \reals\}, \quad t \in \reals.
\]

Set
\[
\mathcal{Z}_{\alpha}:=\{ t \in \reals: X_t \wedge X_{t-}=M_t \}.
\]
It is shown in \cite{zbMATH06288068}[Theorem 2.6] that this set is closed, unbounded above and below, stationary, and regenerative.
We establish in Proposition~\ref{P:Z_alpha_reg_emb} that 
\[
\mathcal{Z}_{\alpha_1} \subseteq \dots \subseteq \mathcal{Z}_{\alpha_n}
\]
for $\vert \E[X_1] \vert < \alpha_1<\dots<\alpha_n$ and that these embeddings are regenerative.
As a consequence, we derive the following result which is proved in Section~\ref{S:Z_alpha_reg_emb}.

\begin{thm}
\lb{lip:ind}
Suppose that $(X_t)_{t \in \reals}$ is a two-sided standard Brownian motion with drift $\beta$
For $\alpha> |\beta|$ set
\[
Y_{\alpha}:=\sup \{t<0 : t \in \mathcal{Z}_{\alpha}\}.
\]
Then $(Y_{\alpha})_{\alpha > |\beta|}$ is a nondecreasing, c\`adl\`ag,  pure jump process with independent increments. 
The point process
\[
\{ (\alpha, Y_{\alpha}-Y_{\alpha-}): Y_{\alpha}-Y_{\alpha-}>0\}
\]
is a Poisson point process on $(|\beta|,\infty) \times (0,\infty)$ with intensity measure
\[
\gamma(ds \times dr)=\frac{\phi(\frac{\sqrt{r}}{s-\beta})+\phi(\frac{\sqrt{r}}{s+\beta})}{\sqrt{r}} \, ds \, dr \; \ind_{\{s>\vert \beta \vert, r>0\}},
\]
where $\phi(x):=\frac{e^{-\frac{x^2}{2}}}{\sqrt{2\pi}}$, for $x>0$.
\end{thm}

\section{Regenerative sets}
\label{S:reg_set}

We introduce the notion of a regenerative set in the sense of \cite{fitztaksar}.
For simplicity, we specialize the definition by only considering random sets defined on probability spaces
(rather than general $\sigma$-finite measure spaces).

\begin{notation}
Let $\Omega^\lra$ denote the class of closed subsets of $\reals$. For $t
\in \mathbb{R}$ and $\omega^\lra \in \Omega^\lra$, define
\[
d_t(\omega^\lra) := \inf \{ s>t: s \in \omega^\lra \}, \quad r_t(\omega^\lra) :=
d_t(\omega^\lra) - t,
\]
and
\[
\tau_t(\omega^\lra) := \mathbf{cl} \{ s-t:s \in \omega^\lra \cap (t,\infty) \}
= \mathbf{cl} \left( (\omega^\lra-t) \cap (0,\infty) \right) .
\]
Here $\mathbf{cl}$ denotes closure and we adopt the
convention $\inf \emptyset = + \infty$. 
Note that 
$t \in \omega^\lra$ if and only if 
$\lim_{s \uparrow t} r_s(\omega^\lra) = 0$, and so
$\omega^\lra \cap (-\infty,t]$ can be reconstructed from
$r_s(\omega^\lra)$, $s \le t$, for any $t \in \mathbb{R}$.
Set
$\mathcal{G}^\lra := \sigma \{ r_s : s \in \reals \}$ and $\mathcal{G}_t^\lra :=
\sigma \{ r_s:s \leq t \}$. Clearly, $(d_t)_{t \in \reals}$ is an
increasing c\`adl\`ag process adapted to the filtration
$(\mathcal{G}_t^\lra)_{t \in \reals}$, and $d_t \ge t$ for all $t \in \reals$.

Let $\Omega^\ra$ denote the class of closed subsets of $\reals_+$.  Define a $\sigma$-field $\mathcal{G}^\ra$ on $\Omega^\ra$ in the same manner that the $\sigma$-field $\mathcal{G}^\lra$ was defined on $\Omega^\lra$.
\end{notation}

\begin{defn}
A \emph{random set} is a measurable mapping $S$ from a measurable
space $(\Omega,\mathcal{F})$ into
$(\Omega^\lra,\mathcal{G}^\lra)$. 
\end{defn}

\begin{defn}
\label{def:regenset}
A probability measure $\Q^\lra$ on $(\Omega^\lra,\mathcal{G}^\lra)$
is regenerative with regeneration law $\Q^\ra$ a probability measure on $(\Omega^\ra,\mathcal{G}^\ra)$ if
\begin{itemize}
  \item[(i)] $\Q^\lra\{d_t = +\infty\} = 0$, for all $t \in \reals$;
  \item[(ii)] for all $t \in \reals$ and
    for all $\mathcal{G}^\lra$-measurable nonnegative functions $F$,
\eq
\lb{eq:regen_law}
\Q^\lra \left [F(\tau_{d_t}) \, | \, \mathcal{G}_{t+}^\lra \right] = \Q^\ra[F],
\en
where we write $\Q^\lra[\cdot]$ and $\Q^\ra[\cdot]$ for expectations
with respect to $\Q^\lra$ and $\Q^\ra$.
\end{itemize}
A random set $S$ defined on a probability space 
$(\Omega,\mathcal{F}, \P)$ is a regenerative set if the
push-forward of $\P$ by the map $S$ (that is,
the distribution of $S$) is a regenerative probability measure.

\end{defn}

\begin{remark}
\label{R:zero_enough}
Suppose that the probability measure $\Q^\lra$ on $(\Omega^\lra,\mathcal{G}^\lra)$
is stationary; that is, if $S^\lra$ is the identity map on $\Omega^\lra$,
then the random set $S^\lra$ on $(\Omega^\lra,\mathcal{G}^\lra, \Q^\lra)$
has the same distribution as $u+S^\lra$ for any $u \in \reals$ or, equivalently,
that the process $(r_t)_{t \in \reals}$ has the same distribution as
$(r_{t-u})_{t \in \reals}$ for any $u \in \reals$.  Then,
in order to check conditions (i) and (ii) of Definition ~\ref{def:regenset},
it suffices to check them for the case $t=0$.
\end{remark}

\begin{thm}
\lb{reg:thm}
The random set $\mathcal{H}_{\alpha}$ is stationary and regenerative.
\end{thm}

\begin{proof}
We first show that $\mathcal{H}_{\alpha}$ is stationary. Let $a \in \reals$. Define the process $(X_{t}^{(a)})_{t \in \reals} := (X_{t-a}-X_{-a})_{t \in \reals}$.  This process is a L\`evy process that has the same distribution as $(X_t)_{t \in \reals}$, and we have
\begin{equation*}
\begin{split}
t \in \mathcal{H}^{X}_{\alpha}+a & \Leftrightarrow t-a \in \mathcal{H}^{X}_{\alpha} \\
& \Leftrightarrow X_{t-a} \wedge X_{(t-a)-}-\alpha(t-a) = \underset{ u \leq t-a}{\text{inf }} (X_u-\alpha u)\\
& \Leftrightarrow X_{t-a}\wedge X_{(t-a)-}-X_{-a}-\alpha(t-a)=\underset{u \leq t}{\text{inf }} (X_{u-a}-X_{-a} -\alpha(u-a) )\\
& \Leftrightarrow X^{(a)}_{t}\wedge X^{(a)}_{t-}-\alpha t = \underset{ u \leq t}{\text{inf }} (X^{(a)}_{u}-\alpha u)\\
& \Leftrightarrow t \in \mathcal{H}^{X^{(a)}}_{\alpha}.
\end{split}
\end{equation*}
Hence, $\mathcal{H}^{X}_{\alpha}+a=\mathcal{H}^{X^{(a)}}_{\alpha}  \ed \mathcal{H}^{X}_{\alpha}$ for all $a \in \reals$, and the stationarity is proved.

Now, because of Remark ~\ref{R:zero_enough}, to prove the regeneration property it suffices to check that conditions (i) and (ii) of Definition~\ref{def:regenset} hold for $t=0$. 
As pointed out in the Introduction, the random set $\mathcal{H}_{\alpha}$ is almost surely unbounded from above, hence condition (i) is verified.

For $t \in \reals$ introduce the random times
\[
D_t:=\inf\{ s > t : s \in \mathcal{H}_{\alpha} \}
= 
\inf\left \{ s >t: X_{s}\wedge X_{s-}-\alpha s = \inf_{u \le s} ( X_u - \alpha u )\right\}
\]
and put
\[
R_t := D_t - t.
\]
It is clear from the d\'ebut theorem that $D := D_0$ is a stopping time with respect to the filtration $(\mathcal{F}_t)_{t \in \reals}$. 
To prove condition (ii), it suffices to show that the random set
\[
\tau_{D}(\mathcal{H}_{\alpha})=\mathbf{cl}\left\{ t >0 : X_{t+D}\wedge X_{(t+D)-}-\alpha (t+D) = \inf_{u \leq t+D} (X_u-\alpha u) \right\}
\]
is independent of of the $\sigma$-field $\bigcap_{\epsilon>0} \sigma \{ R_s : s\leq \epsilon \}$.

We shall prove first that 
\begin{equation}
\label{E:containment1}
\bigcap_{\epsilon>0} \sigma \{ R_s : s\leq \epsilon \} \subseteq \mathcal{F}_D.
\end{equation}

It is clear that
\begin{equation}
\label{E:containment2}
\bigcap_{\epsilon>0} \sigma \{ R_s : s\leq \epsilon \} \subseteq \bigcap_{n \in \nats}  \mathcal{F}_{D_{\frac{1}{n}}}.
\end{equation}

Moreover, for a sequence of nonincreasing stopping times $T_n$ converging almost surely to a stopping time $T$, we have
\begin{equation}
\label{E:filtration_right_continuity}
\bigcap_{n \in \nats}  \mathcal{F}_{T_n} = \mathcal{F}_T.
\end{equation}
To see this, take $\epsilon > 0$ and consider a random variable
$Z$ that is $\bigcap_{n \in \nats}  \mathcal{F}_{T_n}$--measurable.   
We have almost surely the convergence $Z \ind_{\{ T_n \leq T+\epsilon\}} \rightarrow Z$.
Note that
$Z \ind_{\{ T_n \leq T+\epsilon\}}$ is $\mathcal{F}_{T+\epsilon}$--measurable.
Thus $Z$ is $\mathcal{F}_{T+\epsilon}$--measurable.
  It follows from the strong Markov property and the Blumenthal zero--one law that 
\[
\bigcap_{\epsilon > 0} \mathcal{F}_{T+\epsilon} = \mathcal{F}_T
\]
and so $Z$ is $\mathcal{F}_T$--measurable.

In order to establish \eqref{E:containment1}, it follows from \eqref{E:containment2} and \eqref{E:filtration_right_continuity} that it is enough to conclude that 
\begin{equation}
\label{E:D_right_continuity}
D_+:=\lim_{n \rightarrow \infty} D_{\frac{1}{n}}=D, \quad \text{a.s.}
\end{equation}
To see this, suppose to the contrary that $\mathbb{P}\{D < D+\} > 0$. 
We claim that $D>0$ on the event $\{D < D_+\}$.  This is so, because on the event $\{0 = D < D_+\}$ the point $0$ is a right accumulation point of $\mathcal{H}_\alpha$ and then  $D_{\frac{1}{n}}$ must converge to zero, which is not possible.  
On the event $\{0 < D\}$ we have that $D_+ \le D_{\frac{1}{N}} \le D$ as soon as
$N$ is large enough so that $\frac{1}{N} < D$.  Thus, $\mathbb{P}\{D < D+\} = 0$ and \eqref{E:D_right_continuity} holds, implying
that \eqref{E:containment1} also holds.

With \eqref{E:containment1} in hand, it is enough to prove that the set $\tau_D(\mathcal{H}_{\alpha})$ is independent of
 $\mathcal{F}_D$.
Observe that 
\[
\begin{split}
\tau_D(\mathcal{H}_{\alpha}) = \mathbf{cl}&\Bigl\{t > 0 : X_{t+D}\wedge X_{(t+D)-}-X_D-\alpha t \\
& \quad =(X_D\wedge X_{D-}-X_D) \wedge \inf_{0 \le u \le t} (X_{u+D}-X_D-\alpha u) \Bigr\}. \\
\end{split}
\]
Because $D$ is a stopping time, the process $(X_{t+D}-X_D)_{t \geq 0}$ is independent of $\mathcal{F}_D$.  It therefore
suffices to prove that $X_D \leq X_{D-}$ a.s.

Suppose that the event $\{X_D>X_{D-}\}$ has positive probability.  Because $X_0 = X_{0-}$ almost surely,
 $D>0$ on this event.

Introduce the nondecreasing sequence $(D^{(n)})_{n \in \nats}$ of stopping times 
\[
D^{(n)} :=\inf \left \{ t >0: X_t \wedge X_{t-} -\alpha t \leq \inf_{u \le t} (X_u-\alpha u) +\frac{1}{n} \right \}
\]
and put $D^{(\infty)}:=\sup_{n \in \nats} D^{(n)}$.   By Lemma~\ref{deterministic},
\[
D =\inf \left \{ t >0: X_t \wedge X_{t-} -\alpha t \leq \inf_{u \le t} (X_u-\alpha u) \right \},
\]
and so $D^{(\infty)} \le D$.
Because $X$ has c\`adl\`ag paths, for all $n \in \nats$ we have that 
$X_{D^{(n)}} \wedge X_{D^{(n)}-}-\alpha D^{(n)} \leq  \inf_{u \le D^{(n)}} (X_u-\alpha u) +\frac{1}{n}$. 
Sending $n$ to infinity and again using the fact that $X$ has c\`adl\`ag paths, 
we get that $X_{D^{(\infty)}} \wedge X_{D^{(\infty)}-} -\alpha D^{(\infty)} \leq \inf_{u \le D^{(\infty)}} (X_u-\alpha u)$,
and so $D^{(\infty)} \in \mathcal{H}_\alpha$.
By definition of $D$, we conclude that $D^{(\infty)}=D$.

 Suppose we are on the event $\{X_D>X_{D-}\}$ (so that $D>0$) and there is $n \in \nats$ such that $D^{(n)}=D$. 
For all $0<s<D$ we have that $X_s \wedge X_{s-} -\alpha s > \inf_{u \le s} (X_u-\alpha u)+\frac{1}{n}$ so by sending $s\uparrow D$ we get that : $X_{D-}-\alpha D \geq \inf_{u \le D} (X_u-\alpha u)+\frac{1}{n}$, which contradicts  $X_{D-}<X_D$.
Thus $D^{(n)} < D$ on the event $\{X_D>X_{D-}\}$. 
By the quasi-left continuity of $X$ we thus have on the event $\{X_D>X_{D-}\}$ that
\[
X_{D-} = \lim_{n \to \infty} X_{D^{(n)}}= X_D, \quad \text{a.s.},
\]
and so $\P\{X_D>X_{D-}\} = 0$ as claimed.
\end{proof}

\section{Relationship with the set of ladder times}
\lb{S:ladder_times}

\noindent
{\bf Proof of Proposition~\ref{P:ladder_times}}
(i) If $t \in \mathcal{R}_\alpha$, then $X_t-\alpha t =\inf_{u \le t} (X_u-\alpha u)$ and so
$X_t \wedge X_{t-} -\alpha t \le \inf_{u \le t} (X_u-\alpha u)$.  It follows from Lemma~\ref{deterministic} that
$t \in \mathcal{H}_\alpha$.

\noindent
(ii) Because the process $(X_t)_{t \in \reals}$ is right-continuous,  it is clear that $\mathcal{R}_{\alpha}$ is closed from the right;
that is, for every sequence $t_n \downarrow t$ such that $t_n \in \mathcal{R}_{\alpha}$ we have $t  \in \mathcal{R}_{\alpha}$.

\noindent
(iii) As the set $\mathcal{H}_{\alpha}$ is closed and $\mathcal{R}_{\alpha} \subseteq \mathcal{H}_{\alpha}$ we 
certainly have $\mathbf{cl}(\mathcal{R}_{\alpha}) \subseteq \mathcal{H}_{\alpha}$.
We showed in the proof of Theorem~\ref{reg:thm} that $X_D \leq X_{D-}$ a.s. and so $D \in \mathcal{R}_{\alpha}$ a.s. 
By stationarity, $D_t \in \mathcal{R}_{\alpha}$ a.s. for any $t \in \reals$.
Therefore, almost surely for all $r \in \Q$  we have $D_r \in \mathcal{R}_{\alpha}$.  
 Suppose that $t \in \mathcal{H}_{\alpha}$. 
Take a sequence of rationals $\{r_n\}_{n \in \nats}$ such that $r_n \uparrow t$.
Then, for all $n \in \nats$, we have
 $ r_n \leq D_{r_n} \leq t$
and
$D_{r_n} \in \mathcal{R}_{\alpha}$.
It follows that $t \in \mathbf{cl}(\mathcal{R}$ and so $\mathbf{cl}\mathcal{R}_{\alpha})=\mathcal{H}_{\alpha}$.
 
\noindent
(iv)  Take $t \in \mathcal{H}_{\alpha}$ that is not isolated on the right so that there exists a sequence $\{t_n\}_{n \in \nats}$
of point in  $\mathcal{H}_{\alpha}$  such that $t_n \downarrow t$ and $t_n>t$. Consider a sequence $(r_n)_{n \in \nats}$ 
of rational numbers such that for every $n \in \nats$ we have $t \leq r_n \leq t_n$. 
We then have $t \leq r_n \leq D_{r_n} \leq t_n$.  Thus,
 $D_{r_n} \downarrow t$ and, as we have already observed, $D_{r_n} \in \mathcal{R}_{\alpha}$ for all $n \in \nats$.
Since $\mathcal{R}_{\alpha}$ is closed from the right, we must have $t \in \mathcal{R}_{\alpha}$. 
Finally, as the set of points isolated on the right is countable, the set $\mathcal{H}_{\alpha} \setminus \mathcal{R}_{\alpha}$ consists of at most countably many points.
\qed

\section{Regenerative embedding generalities}
\label{S:reg_emb_gen}

We recall the notion of a regenerative embedding of a sequence of regenerative sets from \cite{bertoin99}. 
We modify it slightly to encompass the whole real line instead of the set of nonnegative real numbers. 
For ease of notation we restrict our definition to the case of two sets.  The generalization to a greater number of sets is straightforward.

\begin{defn}
\lb{reg:emb}
Recall that $\Omega^{\leftrightarrow}$  
is the set of closed subsets of $\reals$
and that $\Omega^{\rightarrow}$
is the set of closed subsets of $\reals_+$).
Set
\[
\bar \Omega :
=\{ \omega=(\omega^{(1)},\omega^{(2)}) \in \Omega^{\leftrightarrow} \times \Omega^{\leftrightarrow} : \omega^{(1)} \subseteq \omega^{(2)} \}.
\]
and
\[
\bar{\Omega} ^{\rightarrow} :=\{ \omega=(\omega^{(1)}, \omega^{(2)}) \in \Omega^{\rightarrow} \times \Omega^{\rightarrow} : \omega^{(1)} \subseteq \omega^{(2)} \}.
\]
Write $M^{(1)}(\omega)=\omega^{(1)}$ and $M^{(2)}(\omega)=\omega^{(2)}$ for the canonical projections on $\bar \Omega$, $M=(M^{(1)},M^{(2)})$.
For $t \in \reals$ put
\[
d_t^{(1)}(\omega)=d_t(\omega^{(1)})
\]
and, with a slight abuse of notation,
\[
\tau_t(\omega)=(\tau_t(\omega^{(1)}),\tau_t(\omega^{(2)})).
\]
Denote by $\mathcal{G}_t$ the sigma-field generated by $d_t^{(1)}$, $M^{(1)}\cap(-\infty,d_t^{(1)}]$,
 and $M^{(2)} \cap (-\infty,d_t^{(1)}]$. 
It is easy to check that $(\mathcal{G}_t)_{t \in \reals}$ is a filtration. 
A probability measure $\mathcal{P}$ is called a \textit{regenerative embedding law} with regeneration law $\mathcal{P}^{\rightarrow}$ if for each $t \in \reals$ and each bounded measurable function $f : \bar{\Omega}^\rightarrow \rightarrow \mathbb{R}$
\begin{equation}
\lb{embedd}
\mathcal{P}[f(M\circ \tau_{d_t^{(1)}}) \, \vert \, \mathcal{G}_t]=\mathcal{P}^{\rightarrow}[f(M)] \text{ on } \{d_t^{(1)} < \infty\}.
\end{equation}
We denote such an embedding by the notation $M^{(1)} \prec M^{(2)}$.
\end{defn}

\begin{remark}
\begin{itemize}
\item [(i)]
If under the probability measure $\mathcal{P}$, the canonical pair $(M_1, M_2)$ of random sets is \textit{jointly stationary}, 
in the sense that for all $t \in \reals$ the pair  $(M_1+t,M_2+t)$ has the same distribution as $(M_1, M_2)$, 
then to check that that there is a regenerative embedding it suffices to verify \eqref{embedd} for $t=0$.
\item[(ii)] A similar definition holds for subsets of $\reals_+$ that
contain zero almost surely, which is the version present in \cite{bertoin99}. 
\end{itemize}
\end{remark}

The following theorem follows straightforwardly from the results in \cite{bertoin99}.

\begin{thm}
\lb{ind:emb}
Let:
\[
\mathcal{S}^{(1)} \prec \mathcal{S}^{(2)} \prec \dots \mathcal{S}^{(n)}
\]
be a jointly stationary sequence of subsets of $\reals$ that are regeneratively embedded in the sense of the Definition~\ref{reg:emb}. 
Let $\Phi_{i}$ be the Laplace exponent of the subordinator associated with each $\mathcal{S}^{(i)}$. 
Introduce the measures $\mu_1,\ldots,\mu_{n}$ on $\reals_+$, defined by their Laplace transforms
\[
 \int_{\reals_+} e^{-\lambda x} \, \mu_i (dx) := \frac{\Phi_{i}(\lambda)}{\Phi_{i+1}(\lambda)}, \quad  \lambda>0, \; 1 \le i \le n,
\]
where we adopt the convention $\Phi_{n+1}(\lambda) := \lambda$, $\lambda > 0$.
Put
\[
c_i :=\frac{1}{\mu_i(\reals_+)}=\lim_{\lambda \downarrow 0} \frac{\Phi_{i+1}(\lambda)}{\Phi_{i}(\lambda)}, \quad 1 \le i \le n.
\]
Define the age processes $A_t^{i}$ for each set $\mathcal{S}^{(i)}$ by
\[
A_t^{i} :=\inf \{ s \geq 0: t-s \in \mathcal{S}^{(i)} \}.
\]
Then, for any $t \in \reals$,
\[
(A_t^{1}-A_t^{2}, \ldots, A_t^{n-1}-A_t^{n}, A_t^{n})  \ed c_1\mu_1 \otimes c_2\mu_2 \otimes \cdots \otimes c_n\mu_n.
\]
\end{thm}

\section{A continuous family of embedded regenerative sets}
\lb{S:H_alpha_reg_emb}

\textbf{For this section, we suppose  that $X$ has a Brownian component or infinite L\'evy measure.  That is,
we suppose that $X$ is not a compound Poisson process with drift.  The latter case is trivial to study.}

\begin{lem}
\label{L:H_alpha_increasing}
For 
\[
\E[X_1]<\alpha_1<\alpha_2< \dots <\alpha_n.
\]
we have
\[
\mathcal{H}_{\alpha_1} \subseteq \mathcal{H}_{\alpha_2} \subseteq \dots \subseteq \mathcal{H}_{\alpha_n}.
\]
\end{lem}

\begin{proof}
By part (i) of Lemma~\ref{deterministic},
\[
\mathcal{H}_{\alpha}:=\left\{ t \in \reals : X_{t}\wedge X_{t-}-\alpha t \le \inf_{u \le t} ( X_u - \alpha u ) \right\}.
\]
Hence, if $\E[X_1] < \alpha' < \alpha''$, $t \in \mathcal{H}_{\alpha'}$, and $u \le t$, then
\[
X_t\wedge X_{t-}-\alpha'' t 
\leq X_u-\alpha' u -(\alpha''-\alpha')t
\leq X_u-\alpha' u -(\alpha'' -\alpha')u
=X_u-\alpha'' u, 
\]
so that $t \in \mathcal{H}_{\alpha''}$.  Thus $\mathcal{H}_{\alpha'} \subseteq \mathcal{H}_{\alpha''}$ for $\E[X_1] < \alpha' < \alpha''$.
\end{proof}

\begin{prop}
\label{P:H_alpha_reg_emb}
For $\mathbb{E}[X_1]<\alpha_1<\alpha_2< \cdots <\alpha_n$ we have
\[
\mathcal{H}_{\alpha_1} \prec \mathcal{H}_{\alpha_2} \prec \dots \prec \mathcal{H}_{\alpha_n}.
\]
\end{prop}

\begin{proof}
For ease of notation, we restrict our proof to the case $n=2$. 

By Lemma~\ref{L:H_alpha_increasing} we have $\mathcal{H}_{\alpha_1} \subseteq \mathcal{H}_{\alpha_2}$ when 
$\mathbb{E}[X_1]<\alpha_1<\alpha_2$.

By stationarity, we only need to verify \eqref{embedd} for $t=0$. 
It is clear that 
\[
D_0^{(1)}:=\inf \left\{ s > 0 : X_{s}\wedge X_{s-}-\alpha_1 s= \inf_{u \le s} (X_u-\alpha_1 u)\right\}
\]
 is an $(\mathcal{F}_t)_{t \in \reals}$-stopping time. From the proof of Theorem ~\ref{reg:thm}, we have that almost surely
\[
X_{D_0^{(1)}} \leq X_{D_0^{(1)}-}.
\]
Now $D_0^{(1)} \in \mathcal{H}_{\alpha_2}$ and hence
\[
\begin{split}
 \mathcal{H}_{\alpha_i}\circ \tau_{D_0^{(1)}} 
= 
\mathbf{cl}\Bigl \{ s > 0 : & X_{s+D_0^{(1)}}\wedge X_{s+D_0^{(1)}-}-X_{D_0^{(1)}}-\alpha_i s  \\
& \quad =\inf_{u \le s} \left(X_{u+D_0^{(1)}}-X_{D_0^{(1)}}-\alpha_i u \right)\Bigr\} \\
\end{split}
\]
for $i=1,2$. Now each of
$D_0^{(1)}$, $\mathcal{H}_{\alpha_1} \cap (-\infty, D_0^{(1)}]$,
 and $\mathcal{H}_{\alpha_2}  \cap (-\infty, D_0^{(1)}]$
is $\mathcal{F}_{D_0^{(1)}}$--measurable, so it 
remains to note that
$(X_{s+D_0^{(1)}}-X_{D_0^{(1)}})_{ s \ge 0}$
is independent of $\mathcal{F}_{D_0^{(1)}}$.
\end{proof}

\noindent
{\bf Proof of Theorem~\ref{ind:onesided}}

It is clear that $G$ is nondecreasing. 

As for the right-continuity, consider $\beta > \E[X_1]$ and a sequence $\{\beta_n\}_{n \in \nats}$ with 
$\beta_n \downarrow \beta$ and $\beta_n > \beta$.
Suppose that $G_{\beta+}:=\lim_{n \rightarrow \infty} G_{\beta_n}>G_{\beta}$. 
For any $u \le G_{\beta+} \le G_{\beta_n}$ we have
\[
X_{G_{\beta_n}} \wedge X_{G_{\beta_n}-} - \beta_n G_{\beta_n} \le X_u -\beta_n u.
\]
Taking the limit as $n$ goes to infinity gives
\[
X_{G_{\beta+}} - \beta G_{\beta+} \le X_u-\beta u
\]
and hence
\[
X_{G_{\beta+}}\wedge X_{G_{\beta+}-} - \beta G_{\beta+} \le X_u-\beta u.
\]
It follows from Lemma~\ref{deterministic} that $G_\beta < G_{\beta+} \in \mathcal{H}_\beta$, but this
contradicts the definition of $G_{\beta}$. 

Corollary VI.10 in \cite{bertoin} gives that the Laplace exponent of the subordinator associated with the ladder time set of the process $(\alpha t -X_t)_{t \ge 0}$ (the subordinator is the right-continuous inverse of the local time associated with this set) is 
\[
\Phi_{\alpha}(\lambda)=\exp \left( \int_{0}^{\infty} (e^{-t}-e^{-\lambda t})t^{-1}\mathbb{P} \{X_t \ge \alpha t\} dt \right).
\]
Fix $\mathbb{E}[X_1]<\alpha_1<\alpha_2 < \cdots < \alpha_n$.
Introduce the measures $\mu_1,\ldots,\mu_{n}$ on $\reals_+$, defined by their Laplace transforms
\[
 \int_{\reals_+} e^{-\lambda x} \, \mu_i (dx) := \frac{\Phi_{\alpha_i}(\lambda)}{\Phi_{\alpha_{i+1}}(\lambda)}, \quad  \lambda>0, \; 1 \le i \le n,
\]
where we adopt the convention $\Phi_{\alpha_{n+1}}(\lambda) := \lambda$, $\lambda > 0$.
Put
\[
c_i :=\frac{1}{\mu_i(\reals_+)}=\lim_{\lambda \downarrow 0} \frac{\Phi_{\alpha_{i+1}}(\lambda)}{\Phi_{\alpha_i}(\lambda)}, \quad 1 \le i \le n.
\]
Set $\nu_i = c_i \mu_i$, $1 \le i \le n$, so that 
\begin{equation}
\lb{Laplace1}
\begin{split}
 & \int_{\reals_+} e^{-\lambda x} \, \nu_i (dx) \\
& \quad = \exp \left (- \int_{0}^{\infty} (1-e^{- \lambda t})t^{-1}\mathbb{P}\{ \alpha_i  t \leq X_t \leq \alpha_{i+1}  t \} \, dt \right), \quad 1 \le i \le n-1, \\
\end{split}
\end{equation}
and
\begin{equation}
\lb{Laplace2}
 \int_{\reals_+} e^{-\lambda x} \, \nu_i (dx) = \exp \left (- \int_{0}^{\infty} (1-e^{- \lambda t})t^{-1}\mathbb{P}\{ X_t \ge \alpha_{n}  t \} \, dt \right).
\end{equation}
Then, by Theorem~\ref{ind:emb},
\[
(G_{\alpha_2}-G_{\alpha_1}, \ldots, G_{\alpha_n}-G_{\alpha_{n-1}}, -G_{\alpha_n})  
\ed \nu_1 \otimes \nu_2 \otimes \cdots \otimes \nu_n.
\]

It follows that the process $G$ has independent increments and that $\lim_{\alpha \to \infty} G_\alpha = 0$ almost surely.  
That  $(G_{\alpha})_{\alpha > \E[X_1]}$ is a pure jump process 
and the Poisson description of  $\{ (\alpha, G_{\alpha}-G_{\alpha-}): \; G_{\alpha}-G_{\alpha-}>0\}$ follows
from \eqref{Laplace1}, \eqref{Laplace2}, and standard L\'evy--Khinchin--It\^o theory.
\qed

%

\begin{remark}
Taking the concatenation of the lines with slopes $\alpha$ between $G_{\alpha}$ and $G_{\alpha-}$ for every jump time $\alpha$
constructs the graph of the convex minorant of the L\'evy process $(-X_{t-})_{t \ge 0}$. 
The conclusion of Theorem~\ref{ind:onesided} thus agrees with the study of the convex minorant of a L\'evy process 
carried out in  \cite{pitmanbravo}.
\end{remark}

\section{Another continuous family of embedded regenerative sets}
\lb{S:Z_alpha_reg_emb}

\begin{prop}
\lb{P:Z_alpha_reg_emb}
For $\vert \E[X_1] \vert < \alpha_1<\dots<\alpha_n$, we have that
\[
\mathcal{Z}_{\alpha_1} \prec \dots \prec \mathcal{Z}_{\alpha_n}.
\]
\end{prop} 

\begin{proof}
We shall just prove the result for the case $n=2$. It is very clear that $\mathcal{Z}_{\alpha_1} \subseteq \mathcal{Z}_{\alpha_2}$, as any $\alpha_1$-Lipschitz function is also an $\alpha_2$-Lipschitz function. 
Moreover, the sets $(\mathcal{Z}_{\alpha_1},\mathcal{Z}_{\alpha_2})$ are obviously jointly stationary, and thus it suffices to check the independence condition for $t=0$. 
Note that $D_{\alpha_1} \in \mathcal{Z}_{\alpha_2}$. 
Using  \cite[Lemma 7.2]{1905.07038} gives that
\[
(\mathcal{Z}_{\alpha_1} \circ \tau_{D_{\alpha_1}}, \mathcal{Z}_{\alpha_2}\circ \tau_{D_{\alpha_1}})
\]
is measurable with respect to $\sigma \{ X_{t+D_{\alpha_1}}-X_{D_{\alpha_1}}: t \ge 0 \}$. The same argument yields 
\[
\mathcal{G}_0=\sigma \{ \mathcal{Z}_{\alpha_1} \cap (-\infty,D_{\alpha_1}], \mathcal{Z}_{\alpha_2} \cap (-\infty,D_{\alpha_1}]\} \subseteq \sigma \{ X_t : t \le D_{\alpha_1} \}
\]
An appeal to \cite[Theorem 3.5]{1905.07038} completes the proof.
\end{proof}

\noindent
{\bf Proof of Theorem~\ref{lip:ind}}

As in the proof of  Theorem~\ref{ind:onesided}, it is clear that the process $(Y_\alpha)_{\alpha > |\beta|}$ is nondecreasing and has independent increments.
We leave to the reader the straightforward proof of that this process is c\`adl\`ag. 

We compute the Laplace exponent $\Phi_{\alpha}$ of the subordinator associated with the regenerative set $\mathcal{Z}_{\alpha}$. 
From \cite[Proposition 8.1]{zbMATH06288068} we have
\begin{align*}
\Phi_\alpha(\lambda)=\frac{4(\alpha^2-\beta^2)\lambda}{(\sqrt{2\lambda+(\alpha-\beta)^2}+\alpha-\beta)(\sqrt{2\lambda+(\alpha+\beta)^2}+\alpha+\beta)}.
\end{align*}
Thus, for $\vert \beta \vert < \alpha_1<\alpha_2$, we have
\begin{align*}
\E[e^{-\lambda(Y_{\alpha_2}-Y_{\alpha_1})}]&=c \frac{\Phi_{\alpha_1}(\lambda)}{\Phi_{\alpha_2}(\lambda)}\\
&=c \frac{(\sqrt{2\lambda+(\alpha_2-\beta)^2}+\alpha_2-\beta)(\sqrt{2\lambda+(\alpha_2+\beta)^2}+\alpha_2+\beta)}{(\sqrt{2\lambda+(\alpha_1-\beta)^2}+\alpha_1-\beta)(\sqrt{2\lambda+(\alpha_1+\beta)^2}+\alpha_1+\beta)},
\end{align*}
where
\[
c=\lim_{\lambda \downarrow 0} \frac{(\sqrt{2\lambda+(\alpha_1-\beta)^2}+\alpha_1-\beta)(\sqrt{2\lambda+(\alpha_1+\beta)^2}+\alpha_1+\beta)}{(\sqrt{2\lambda+(\alpha_2-\beta)^2}+\alpha_2-\beta)(\sqrt{2\lambda+(\alpha_2+\beta)^2}+\alpha_2+\beta)};
\]
that is,
\[
c=\frac{\alpha_1^2-\beta^2}{\alpha_2^2-\beta^2}.
\]
Hence,
\[
\log\left(\E\left[e^{-\lambda(Y_{\alpha_2}-Y_{\alpha_1})}\right]\right)=f(a_3)+f(a_4)-f(a_1)-f(a_2),
\]
where $a_1=(\alpha_1+\beta)^{-1}$, $a_2=(\alpha_1-\beta)^{-1}$, $a_3=(\alpha_2+\beta)^{-1}$ and $a_4=(\alpha_2-\beta)^{-1}$, and 
\[
f(x)=-\log(1+\sqrt{2\lambda x^2+1}).
\]
It remains to observe that
\[
f(x)=-\int_{0}^{\infty}(1-e^{-\lambda r})r^{-\frac{1}{2}}\int_{0}^x t^{-2}\phi(t\sqrt{r}) \, dt \, dr
\]
and do a change of variable inside the integral to finish the proof.
\qed

\section{Some real analysis}

\begin{lem}
\lb{deterministic}
Fix a c\`adl\`ag function $f :\mathbb{R} \mapsto \mathbb{R}$ and consider the set
\[
\mathcal{H} :=\{ t \in \reals : f(t)\wedge f(t-)=\inf_{u \le t} f(u)  \}.
\]
\begin{itemize}
\item[(i)]
The set $\mathcal{H}$ coincides with
\[
\{ t \in \reals : f(t)\wedge f(t-) \le \inf_{u \le t} f(u)  \}.
\]
\item[(ii)]
The set $\mathcal{H}$ is closed. 
\item[(iii)]
If $\lim_{t \rightarrow -\infty} f(t)=+\infty$ and $\lim_{t \rightarrow +\infty} f(t)=-\infty$, then the set $\mathcal{H}$ is nonempty and unbounded from above and below. 
\end{itemize}
\end{lem}

\begin{proof} 
(i) Note that $\{ t \in \reals : f(t)\wedge f(t-) \le \inf_{u \le t} f(u)  \}$ is the disjoint union
$\{ t \in \reals : f(t)\wedge f(t-) = \inf_{u \le t} f(u)  \} \sqcup \{ t \in \reals : f(t)\wedge f(t-) < \inf_{u \le t} f(u)  \}$.
Clearly, $ f(t)\wedge f(t-) \ge \inf_{u \le t} f(u)$ for all $t \in \reals$ and so the second set on the right hand side is empty.

\noindent
(ii) We want to show that if $\{t_n\} _{n \in \nats}$ is a sequence of elements of $\mathcal{H}$ converging to some $t^* \in \reals$, then $t^* \in \mathcal{H}$. The result is clear if $t_n = t^*$ infinitely often, so we may suppose that $t^* \notin \{t_n\} _{n \in \nats}$.

Suppose to begin with that there are only finitely many $n \in \mathbb{N}$ such that $t_n<t^*$. Then, for $n$ large enough, we have that $t_n > t^*$ and thus $f(t_n)\wedge f(t_n-) \leq f(u)$ for all $u \leq t^*$. 
Now $\lim_{n \to \infty} f(t_n) = \lim_{n \to \infty} f(t_n-) = f(t^*)$.  Hence, $f(t^*) \wedge f(t^*-) \leq f(t^*) \le f(u)$ for all $u \le t^*$ and so $t^* \in \mathcal{H}$ by part (i). 

Suppose  on the other hand, that the set $\mathcal{N}$ of $n \in \mathbb{N}$ such that $t_n<t^{*}$ is infinite.  For $u < t^*$ we have for large $n \in \mathcal{N}$ sufficiently large that $u \le t_n$ and thus $f(t_n) \wedge f(t_n-) \leq f(u)$. Now the limit as $n \to \infty$ with $n \in \mathcal{N}$ of $f(t_n)\wedge f(t_n-)$ is $f(t^*-)$.  Hence,  $f(t^*)\wedge f(t^*-) \le f(t^*-) \leq \inf_{u < t^*} f(u)$.  This implies that $f(t^*)\wedge f(t^*-) \leq \inf_{u \le t^*} f(u)$ and so $t^* \in \mathcal{H}$ by part (i).

\noindent
(iii)  Fix $M \in \reals$, put $I = \inf_{t \le M} f(t)$, and let $\{t_n\}_{n \in \nats}$ be a sequence of elements of $(-\infty, M]$ such that $\lim_{n \to \infty} f(t_n) = I$.  Because $f(t)$ goes to $+\infty$ as $t\rightarrow -\infty$, the sequence $\{t_n\}_{n \in \nats}$ is bounded and thus admits a subsequence $\{t_{n_k}\}_{k \in \nats}$ that converges to some $t^* \in (-\infty, M]$. By the argument in part (ii), $I \in \{f(t^*), f(t^*-)\}$.  Moreover, $I \le f(t^*)$ and $I \le f(t^*-)$. Thus, $f(t^*)\wedge f(t^*-) = I = \inf_{u \le M} f(u) \le \inf_{u \le t} f(u)$ and $t \in \mathcal{H}$ by part (i). Since $M \in \reals$ is arbitrary it follows that $\mathcal{H}$ is not only nonempty but also unbounded below.

Because  $f(t)$ goes to $+\infty$ as $t\rightarrow -\infty$ and  $f(t)$ goes to $-\infty$ as $t\rightarrow +\infty$, for each $n \in \nats$ we have that the set $\{t \in \reals : f(t) \le -n\}$ is nonempty and bounded below and so  $s_n := \inf\{t \in \reals : f(t) \le -n\} \in \reals$.  The sequence $\{s_n\}_{n \in \nats}$ is clearly nondecreasing and unbounded above.  Now $f(s_n) \wedge f(s_n-) = f(s_n) = \inf\{f(u) : u \le s_n\}$ for all $n \in \nats$
so that $s_n \in \mathcal{H}$ for all $n \in \nats$ and hence $\mathcal{H}$ is unbounded above.  
\end{proof}

\providecommand{\bysame}{\leavevmode\hbox to3em{\hrulefill}\thinspace}
\providecommand{\MR}{\relax\ifhmode\unskip\space\fi MR }
\providecommand{\MRhref}[2]{%
  \href{http://www.ams.org/mathscinet-getitem?mr=#1}{#2}
}
\providecommand{\href}[2]{#2}

\end{document}